\documentclass[11pt]{amsart}

\usepackage{amsmath, hyperref}
\usepackage{color}

\def\R{{\mathbb {R}}}
\def\N{{\mathbb {N}}}

\def\p{{\mathbf {p}}}

\def\div{\operatorname {\text{div}}}
\def\dist{\operatorname {\text{dist}}}

\newtheorem{teo}{Theorem}[section]
\newtheorem{lema}[teo]{Lemma}
\newtheorem{prop}[teo]{Proposition}

\theoremstyle{remark}
\newtheorem{remark}[teo]{Remark}

\theoremstyle{definition}
\newtheorem{defi}[teo]{Definition}

\numberwithin{equation}{section}

\title[A priori estimates]{A priori estimates for solutions of $g$-Laplace type problems}
\author[I. Ceresa-Dussel, J. Fern\'andez Bonder and A. Silva]
{Ignacio Ceresa-Dussel, Juli\'an Fern\'andez Bonder and Anal\'{\i}a Silva}

\address{Juli\'an Fern\'andez Bonder \hfill\break\indent
Departamento  de Matem\'atica, FCEyN, Universidad de Buenos Aires, \hfill\break\indent 
Instituto de C\'alculo, CONICET\hfill\break\indent
Ciudad Universitaria, 0+$\infty$ building, C1428EGA, Av. Cantilo s/n\hfill\break\indent
Buenos Aires, Argentina}

\email{{\tt jfbonder@dm.uba.ar}\hfill\break\indent {\it Web page:} {\tt http://mate.dm.uba.ar/$\sim$jfbonder}}

\address{Anal\'{\i}a Silva \hfill\break\indent
Departamento de Matem\'atica, FCFMyN, Universidad Nacional de San
Luis \hfill\break\indent Instituto DE Matem\'atica Aplicada San
Luis, IMASL, CONICET. \hfill\break\indent Italia avenue 1556, San Luis (5700), 
\hfill\break\indent
San Luis, Argentina.}

\email{{\tt acsilva@unsl.edu.ar}\hfill\break\indent {\it Web page:} {\tt https://analiasilva.weebly.com/}}

\address{Ignacio Ceresa Dussel\hfill\break\indent
Departamento  de Matem\'atica, FCEyN, Universidad de Buenos Aires, \hfill\break\indent 
Instituto de C\'alculo, CONICET\hfill\break\indent
Ciudad Universitaria, 0+$\infty$ building, C1428EGA, Av. Cantilo s/n\hfill\break\indent
Buenos Aires, Argentina}

\email{{\tt iceresad@dm.uba.ar}}

\thanks{Supported by Universidad de Buenos Aires under grant X078, by ANPCyT PICT No. 2006-290 and CONICET (Argentina) PIP 5478/1438. J. Fern\'andez Bonder is a member of CONICET. Analia Silva is a fellow of CONICET}



\begin{document}

\begin{abstract}
In this work we study a priori bounds for weak solution to elliptic problems with nonstandard growth that involves the so-called $g-$Laplace operator.  The $g-$Laplacian is a generalization of the $p-$Laplace operator that takes into account different behaviors than pure powers. The method to obtain this a priori estimates is the so called ``blow-up'' argument developed by Gidas and Spruck. Then we applied this a priori bounds to show some existence results for these problems.
\end{abstract}

\maketitle

\section{Introduction}
In this work we consider the problem 
\begin{equation}\label{P}
\begin{cases}
\Delta_g u + B(x,u, \nabla u) = 0 & \text{in }\Omega,\\
u>0 & \text{in }\Omega,\\
u=0 & \text{on }\partial\Omega,
\end{cases}
\end{equation}
where $\Omega$ is a smooth domain of $\R^n$ ($n\geq2$), $B$ is the non linear source term and $\Delta_g$ is the $g$-Laplace operator, that is defined as
$$
\Delta_g u := \text{div}\left(g(|\nabla u|) \frac{\nabla u}{|\nabla u|}\right).
$$ 
Recall that when the nonlinearity $g(t)=t^{p-1}$ this operator becomes the well-known $p-$Laplace operator that has been widely studied in the literature.
The use of more general nonlinearities $g(t)$ comes from the fact that in several applications it is needed to consider growth laws different from pure powers or different behaviors near zero and near infinity.

For problems of the form \eqref{P} with general nonlinearities $g(t)$, the use of Orlicz and Orlicz-Sobolev spaces provide a natural framework for analysis. These spaces have been the subject of extensive research since the 1950s and are by now well understood. For a general introduction to Orlicz  and Orlicz-Sobolev spaces, interested readers may refer to the book of M. A. Krasnoselskii \cite{Krasnoselskii}.  

Equations like \eqref{P}  have been applied in various fields of sciences such as physics \cite{P1,P2}, ecology \cite{E1,E2}, image processing \cite{IP1}, and fluid dynamics \cite{FD1}. The study of the regularity for bounded solutions to \eqref{P} was carried on in the seminal paper of G. Lieberman \cite{Lieberman}. The results of \cite{Lieberman} will be of crucial use in this work.

We are interested in two aspects of problem \eqref{P}. Firstly, we consider the existence of a priori bounds for all weak solutions of \eqref{P}. This subject was considered in the seminal paper of Gidas and Spruck \cite{Gidas-Spruck}, where they  derive a priori bounds for positive solutions of the semi-linear elliptic boundary value problem
$$
\begin{cases}
\frac{\partial}{\partial x_j}(a^{ij}(x)\frac{\partial u}{\partial x_j})+b_j(x)u_{x_j}(x)+f(x,u)=0&\text{ in } \Omega,\\
u(x)=0 &\text{ on } \partial\Omega.
\end{cases}
$$

 In \cite{Gidas-Spruck}, the authors used a ``blow up" argument which reduces the problem to a global nonexistence results of Liouville type. The technique developed in \cite{Gidas-Spruck} was then replicated by several authors to cover more general equations and systems. See for instance \cite{DRuiz, Zou, Clement-Mansevich, Hu}.

In \cite{DRuiz}, D. Ruiz studied the existence of positive solutions for a nonlinear Dirichlet problem involving the $p$-Laplacian, when the non-linearity depend on $x, u$ and $\nabla u$ . 
$$
\begin{cases}
	\Delta_p u+B(x,u,\nabla u)=0&\text{ in }\Omega,\\
	u=0 &\text{ on }\partial\Omega.
\end{cases}
$$
However, the author was limited to consider the case where $1<p\le 2$, since at that time, there were no available Liouville-type results for the generic $p-$Laplacian in the halfspace. This restriction was overcome by H. Zou in \cite{Zou}, who prove the necesary Liouville type non-existence theorem on the half-space $\R^n_+$, which is needed to use the standard
blow-up device. 

The first objective of this work is to provide with the generalization of H. Zou's results to problems of the form \eqref{P}. That is, to prove uniform a priori bounds for nonnegative weak solutions of \eqref{P}.

In order to obtain such results, we use the classical blow-up argument of Gidas-Spruck, and the key observation is that in the limit of the blow-up procedure, under mild hypotheses on the function $g(t)$, the same Liouville-type theorems for the $p-$Laplace operator used in \cite{Zou} are needed.

The work of H. Zou \cite{Zou} is the main source of inspiration for this paper.

Secondly, we aim to prove the existence of solutions to \eqref{P}. In general, when  $B$
depends on $\nabla u$, variational methods cannot be applied to deal with \eqref{P}. Hence, the apriori bounds estimates and a fixed point theorem  are crucial. 

So the second main result of this work is to show the existence of nontrivial, nonnegative solutions to \eqref{P} by means of a fixed point argument, where the a priori bounds for weak solutions are the key ingredient.

To end this introduction we want to point out that recently in \cite{Barletta-Tornatore, Barletta-Tornatore2, Barletta} some existence results for equations of the type \eqref{P} were proved by different arguments. In fact, the authors use the method of sub and supersolutions to show the existence of weak solutions to \eqref{P}.

\subsection*{Organization of the paper}
The paper is structured as follows: in Section \ref{Preliminares}, we give the framework of our work and recall some results on $g-$Laplace equations that will be used in the rest of the paper.
In Section \ref{Apriori}, we state and prove the existence of a priori bounds of solution of \eqref{P}. Finally in Section \ref{Existencia} we prove the existence results for \eqref{P}.

\section{Preliminaries}\label{Preliminares}
In this section we will recall definitions and preliminary results on Orlicz and Orlicz-Sobolev spaces and then we introduce the $g-$Laplace operator and recall its main properties and the regularity results for solutions of \eqref{P} needed in this work. Finally we will state some Liouville type theorems for the $p-$Laplace operator that will turn out to be crucial in our arguments.

\subsection{Young functions}
We define a Young function as an application $G\colon [0,\infty)\to [0,\infty)$, that is increasing, convex, $G(0)=0$ and of class $C^1$. Let us denote $g(t)=G'(t)$

The function $g(t)$ will be assume to verify that
$$
g(t) = \int_0^t g'(s)\, ds,
$$
where $g'$ is a nonnegative, right-continuous, locally integrable function.

In fact, we will further assume  that $g$ satisfies the so-called Lieberman conditions, i.e.
\begin{equation}\label{Lieberman.condition}
	p^--1\le \frac{g'(t)t}{g(t)}\le p^+-1,
\end{equation}
for some $1<p^-\le p^+<\infty$.
 
From this inequality it can be easily verified that
 \begin{equation}\label{Lieberman2.condition}
 	p^-\le \frac{tg(t)}{G(t)}\le p^+.
 \end{equation}

A Young function $G$ is said to verify the $\Delta_2-$condition if there exists a constant $C>0$ such that
\begin{equation}\label{Delta2.condition}
G(2t)\leq CG(t),\quad t\geq0.
\end{equation}

Condition \eqref{Lieberman2.condition} ensures that both $G$ and $\tilde G$ satisfy the $\Delta_2$ condition, where $\tilde G$ denotes the complementary function of $G$, that is defined as 
$$
\tilde G(t)  :=\sup\{t\omega-G(\omega):\omega>0\}.
$$
For a proof of all the assertions in this subsection and a thorough introduction to the subject, we refer to \cite{Krasnoselskii}.

Throughout this article $G$ will always denote a Young function that verifies \eqref{Lieberman.condition} and $G':=g$ will always be assume to have {\em regular variation at infinity}, i.e.
 \begin{equation}\label{reg.var}
 \lim_{s\to\infty} \frac{g(st)}{g(s)} = t^{p-1},
 \end{equation}
 uniformly on bounded intervales $t\in [0,T]$.
 
 Observe that one can easily check that $p^-\le p\le p^+$. This exponent $p$ can be thought as the {\em exponent at infinity} for the function $g$.

Some examples of Young functions $G$ satisfying \eqref{Lieberman.condition}--\eqref{reg.var} includes the most common uses of Young functions, for instance $G(t) = t^{p}\ln^\alpha(t + 1)$, $p>1$, $\alpha>0$ and $G(t)=\frac1p t^p + \frac1q t^q$, $1<p,q<\infty$.

\subsection{Orlicz spaces}
Given $G\colon \R_+ \to \R_+$ a Young function, we consider the spaces
$$
L^G(\Omega)=\left\{u\in L^1_\text{loc}(\Omega)\colon \Phi_G(u)<\infty\right\}
$$
and
$$
W^{1,G}(\Omega)=\left\{u\in L^G(\Omega)\colon \Phi_{1,G}(u)<\infty \right\},
$$
where
$$
\Phi_G(u) = \Phi_{G,\Omega}(u) =\int_{\Omega} G(|u|)\,dx\quad\text{ and }\quad\Phi_{1,G}(u)= \Phi_{1,G,\Omega}(u) = \int_{\Omega}G(|\nabla u|)\,dx.
$$
These spaces are endowed, respectively, with the Luxemburg norms defined as

$$
\|u\|_G=\|u\|_{G,\Omega} = \|u\|_{L^G(\Omega)}=\inf\left\{ \lambda>0 \colon \Phi_G\left(\frac{u}{\lambda}\right)\leq 1\right\}
$$
and 
$$
\|u\|_{1,G}=\|u\|_{1,G,\Omega} = \|u\|_{W^{1,G}(\Omega)}=\|u\|_G+ \inf\left\{ \lambda>0 \colon \Phi_{1,G}\left(\frac{u}{\lambda}\right)\leq 1\right\}.
$$
In the following proposition, we recall some properties of these spaces.
\begin{prop}\cite[Chapter 8]{Adams}
	Let $G$ be a Young function that satisfies the condition \eqref{Lieberman2.condition}. Then, the spaces $L^G(\Omega)$ and $W^{1,G}(\Omega)$ are reflexive, separable Banach spaces. Moreover, the dual space of $L^G(\Omega)$ can be identified with $L^{\widetilde{G}}(\Omega)$. Finally, $C^\infty_c(\R^n)$ is dense in both $L^G(\R^n)$ and $W^{1,G}(\R^n)$.
\end{prop}

\subsection{The $g-$Laplace operator}
Given $\Omega \subset \R^n$ an open set, we define 
$$
		W^{1,G}_0(\Omega) = \overline{C^\infty_c(\Omega)},
$$
where the closure is taken with respect to the norm $\|\cdot\|_{1,G}$.
Therefore, the topological dual space of $W^{1,G}_0(\Omega)$ is contained in the space of distributions $\mathcal D'(\Omega)$ and it will be denoted by $W^{-1,\tilde G}(\Omega) $.

Then, we present the $g-$Laplace operator $-\Delta_g\colon W^{1,G}_0(\Omega)\to W^{-1,\tilde G}(\Omega)$ as
$$
\langle-\Delta_g u,v\rangle=\int_\Omega g(|\nabla u|)\frac{\nabla u}{|\nabla u|}\nabla v\,dx.
$$
Where $\langle \cdot , \cdot \rangle$  denotes the duality pairing between $W^{1,G}_0(\Omega)$ and $ W^{-1,\tilde G}(\Omega)$.
We need the definition of a weak solution to \eqref{P}.
\begin{defi}
	A function $u\in W^{1,G}_0(\Omega)\cap L^\infty(\Omega)$ is said to be a weak solution of \eqref{P} if
	$$
	\int_\Omega g(|\nabla u|)\frac{\nabla u}{|\nabla u|}\nabla v\,dx=\int_\Omega B(x,u,\nabla u)v\,dx
	$$
	for all $v\in W^{1,G}_0(\Omega)\cap L^\infty(\Omega)$.
\end{defi}

The regularity theory for weak solutions of \eqref{P} was establishedby G. Lieberman in \cite{Lieberman}. In fact, the author in \cite{Lieberman} analyzed slightly more general problems than \eqref{P} and when we specialized  \cite[Theorem 1.7]{Lieberman} to \eqref{P} we obtain the following regularity theorem.

\begin{teo}\label{Adap.Lieberman}
	Let $\Omega$ be a bounded domain in $\R^n$ with $C^{1,\alpha}$ boundary for some $0<\alpha\leq1$. Suppose that $G$ is a Young function satisfying \eqref{Lieberman.condition}, and consider the problem \eqref{P}. Assume that $B$ satisfies
	$$
	|B(x,t,\p)|\leq K(1+g(\p)\p)
	$$	
	for some positive constants $K$ and $M_0$, all $x\in\Omega$, all $t\in[-M_0,M_0]$, and all $\p\in\R^n$. Then, any weak solution $u\in W^{1,G}_0(\Omega)$ with $|u|\leq M_0$ in $\Omega$ is $C^{1,\beta}(\overline{\Omega})$ for some positive $\beta$.
	
\end{teo}
\subsection{Liouville non-existence type theorem}
Finally, we state a Liouville non-existence type theorem for the $p-$Laplacian on the half-space $\R^n_+$ and on the entire space $\R^n$, when $B$ depends only on $u$. This result is a particular case of \cite[Theorem 1.1]{Zou}.

\begin{teo}\label{Liouville} Given $1\leq p<n$ and $p^*=\frac{np}{n-p}$.
	Assume that $B(x,t,\p)=B(t)$ is continuously differentiable for $t>0$ and that there exist positive constants $K>0$, $q\in (p, p^*)$ and $r\in(0,p^*-1)$ such that for any $t>0$
	$$
	K^{-1}t^{q-1}\leq B(t)\leq Kt^{q-1},\quad rB(t)\geq tB'(t).
	$$ 
	Then 
\begin{equation}\label{eq.liouville}
	\begin{cases}
		\Delta_p u+B(u)=0 &\text{ in } \Omega \\
		u > 0 &\text{ in } \Omega\\
		u=0 &\text{ on } \partial \Omega
	\end{cases}
\end{equation}
	 does not admit  any non-negative non-trivial solution $u$ in the half space $\Omega=\R^n_+$ or in the entire space $\Omega = \R^n$.
\end{teo}

\section{A priori estimates}\label{Apriori}
In this section we prove the existence of a priori $L^\infty(\Omega)$ bound for solutions of the quasi-linear elliptic differential equation 
\begin{equation}
\begin{cases}
\Delta_g u + B(x,u, \nabla u) = 0 & \text{in }\Omega\\
u>0 & \text{in }\Omega\\
u=0 & \text{on }\partial\Omega,
\end{cases}
\end{equation}
where $\Omega\subset \R^n$ ($n\ge2$) is a bounded smooth domain.

The nonlinear source term $B(x,t,\p)$ verifies the \textbf{growth condition}

	\begin{equation}\label{cota.B}
		|B(x,t,\p)|\le K(1 + f(t) + h(|\p|)),
		\end{equation}
for some constant $K>0$,	and some functions $f,h\colon \R_+\to \R_+$.

The function $f$ is assume to verify that $t f(t)\gg G(t)$ in the sense that for any $C>0$, there exists $t_0>0$ such that
\begin{equation}\label{f<G}
	t f(t)\ge G(Ct),
\end{equation}
for every $t>t_0.$
	
	On $h$ we assume that given $s_0>0$, there exists $C>0$ such that
	\begin{equation}\label{eti1}
	\frac{h(G^{-1}(s f(s)) t)}{f(s)} \le C(1+G(t)),
	\end{equation}
	for every $s>s_0$.
	
%
	
Furthermore, we assume that $B$ satisfies the following \textbf{limit condition}: there exists $q>1$ and a continuous function $b:\overline{\Omega}\to \R$ such that for every $(M_k,a_k)$ with $a_k = O(G^{-1}(M_k f(M_k)))$, we have
	\begin{equation}\label{lim.B}
		\lim_{k\to\infty}\frac{B(x,M_k t,a_k \p)}{f(M_k)} = b(x)t^{q-1},
	\end{equation}
	uniformly in $\Omega$ .

Observe that if
$$
B(x,t,\p) = Af(t) + Bf_0(t) + C h(|\p|),
$$
where $f$ is of regular variation at infinity, $f_0\ll f$, and $h$ verifies \eqref{eti1},
then $B$ satisfy ours growth and limit conditions.

Now, we are able to prove the a priori estimates for \eqref{P}. 
\begin{teo}\label{teo.apriori}
Let $\Omega\subset \R^n$ be a bounded domain with $C^{1,\beta_0}$ boundary. Let $u\in W^{1, G}_0(\Omega)\cap C(\overline{\Omega})$ be a weak solution to
\begin{equation}\label{P.lambda}
\begin{cases}
\Delta_g u + B(x,u, \nabla u) + \lambda = 0 & \text{in }\Omega\\
u>0 & \text{in }\Omega\\
u=0 & \text{on }\partial\Omega,
\end{cases}
\end{equation}
where $G$ is a Young functions satisfying \eqref{Lieberman.condition} and \eqref{reg.var}. Assume moreover that $B$ satisfies the limit condition \eqref{lim.B} and the growth condition \eqref{cota.B}, with the assumptions that the function $f$ satisfies \eqref{f<G}, and $h$ satisfies \eqref{eti1}.  Moreover, assume that $p$, as determined by \eqref{reg.var}, and $q$, as determined by \eqref{lim.B}, are exponents such that $q\in (p,p^*)$.

Then, there exists a constant $C>0$ such that
$$
\|u\|_\infty + \lambda \le C.
$$
where the value of $C$ depends on $\Omega$, $K$, $p^+$, $p^-$. The constant $K$ is determined by equation \eqref{cota.B}, and the constants $p^+$, $p^-$ were determined in \eqref{Lieberman.condition}.
\end{teo}

\begin{proof}
Assume that there exists a sequence $\{(u_k, \lambda_k)\}_{k\in\N}$ of weak solutions to \eqref{P.lambda} such that
$$
\|u_k\|_\infty + \lambda_k\to\infty. 
$$
Let us defined
$$
M_k := \sup_{x\in\Omega} u_k(x) = u_k(x_k),
$$
with $x_k\in \overline{\Omega}$ and let $\phi\colon\R_+\to\R_+$ be the function given implicitly by
$$
\phi(t)g(t\phi(t))=f(t).
$$
It is easy to see that $\phi$ is a well-defined, continuous and nondecreasing function such that $\phi(0)=0$ and $\phi(\infty)=\infty$.

Furthermore, given $\varphi$ as
\begin{equation}\label{phitilde}
\varphi(t) := \frac{G^{-1}(tf(t))}{t},
\end{equation}
 $\phi$ and $\varphi$ verifies 
\begin{equation}\label{comparison.phi.phitilde}
k_-\varphi(t)\leq \phi(t)\leq k_+\varphi(t),
\end{equation}
where the constants $k_\pm$ depend only on $p_\pm$.
Observe that since $tf(t)\gg G(t)$, we have that $\varphi(t)\to\infty$ as $t\to \infty$ and hence $\phi(t)\to\infty$ as $t\to\infty$.

Next, let $N_k>0$ and $y_k\in \overline{\Omega}$ be given and define the rescaled functions
$$
v_k(x) = \frac{1}{N_k} u_k\left(y_k + \frac{x}{\phi(N_k)}\right),
$$
and the rescaled domains
$$
\Omega_k := \left\{x\in \R^n\colon y_k + \frac{x}{\phi(N_k)}\in \Omega\right\}.
$$
Now, direct computations gives that $v_k$ satisfies in $\Omega_k$,
$$
\phi(N_k)\div\left(g\left(N_k\phi(N_k)|\nabla v_k|\right)\frac{\nabla v_k}{|\nabla v_k|}\right) + B\left(y_k + \frac{x}{\phi(N_k)}, N_k v_k, N_k\phi(N_k)\nabla v_k\right) +\lambda_k= 0.
$$
Denoting
$$
g_k(t) := \frac{g\left(N_k\phi(N_k) t\right)}{g\left(N_k\phi(N_k)\right)},
$$
and
$$
B_k(x,t,\p) :=  \frac{B\left(y_k + \frac{x}{\phi(N_k)}, N_k t, N_k\phi(N_k)\p\right)}{f(N_k)}, \quad \mu_k := \frac{\lambda_k}{f(N_k)},
$$
then using \eqref{phitilde} it is easy to see that $v_k$ is a weak solution to
\begin{equation}\label{Pk}
\begin{cases}
\Delta_{g_k} v_k + B_k(x,v_k,\nabla v_k) + \mu_k= 0& \text{in } \Omega_k,\\
v_k>0\\
v_k=0 & \text{on }\partial\Omega_k.
\end{cases}
\end{equation}

 Therefore, using \eqref{cota.B} \eqref{phitilde} and \eqref{comparison.phi.phitilde} we get
\begin{equation}\label{cota.Bk}
|B_k(x,t,\p)|\le K\left(1 + \frac{f(N_k t)}{f(N_k)} + \frac{h(k_+G^{-1}(N_k f(N_k)) \p)}{f(N_k)}\right). 
\end{equation}

Next, the proof is divided into two cases depending on the behavior of $f(M_k)/\lambda_k$.

\underline{Case 1}: $f(M_k)/\lambda_k$ is unbounded.

In this case, we can assume that
$$
\lim_{k\to\infty} \frac{f(M_k)}{\lambda_k} = \infty \Longrightarrow M_k\to\infty.
$$
In this case, we take
$$
N_k = M_k \qquad \text{and}\qquad y_k=x_k
$$
therefore $\|v_k\|_\infty = v_k(0)=1$.

Then, from our conditions on $f$ and \eqref{eti1} and using \eqref{cota.Bk} we readily obtain
$$
|B_k(x,v_k,\nabla v_k)|\le \tilde K(1+ G(|\nabla v_k|)),
$$
and observe that in this case, $\mu_k\to 0$. In particular, $\mu_k$ is bounded.

Since the transformation $x_k + \frac{x}{\phi(M_k)}$ flattens the boundary, we have that
$$
\|\partial\Omega_k\|_{1,\beta_0}\le \|\partial\Omega\|_{1,\beta_0}
$$
(see \cite{Zou} for the details).

Then by the regularity estimates of Theorem \ref{Adap.Lieberman}, we get that there exists a constant $C$ independent of $k$ such that
\begin{equation}\label{est.vk}
\|v_k\|_{C^{1,\beta}(\overline{\Omega_k})}\le C
\end{equation}
for some $\beta>0$ also independent of $k$.

Observe that, since $v_k(0)=1$, $v_k=0$ on $\partial\Omega_k$ and $|\nabla v_k|\le C$ for every $k$, there exists a $\rho>0$ such that
$$
\dist(0,\partial\Omega_k)\ge \rho,\quad\text{for every }k\in\N.
$$
This case now breaks down into two subcases, either $\dist(0,\partial\Omega_k)$ is bounded or not.

\underline{Subcase 1.1}: Assume that  $\dist(0,\partial\Omega_k)$ is unbounded. Then, we can assume that
$$
\dist(0,\partial\Omega_k)\to\infty.
$$
In this case, we have that $\Omega_k\to\R^n$ in the sense that, given $R>0$, $B_R(0)\subset\Omega_k$ for $k$ large.

Next, from \eqref{est.vk}, using Arzela-Ascoli's Theorem together with a diagonal argument, we have that there exists $v\in C^{1,\beta/2}(\R^n)$ such that (passing to a subsequence, if necessary)
$$
v_k\to v \qquad \text{in } C^{1,\beta/2}_{loc}(\R^n).
$$
By our limits assumptions on $B$ \eqref{lim.B}, it follows that
$$
B_k(x,v_k,\nabla v_k)\to b(x_0)v^{q-1}\quad \text{uniformly on compact sets of }\R^n,
$$
where $x_0 = \lim_{k\to\infty} x_k$.

Moreover, by \eqref{reg.var} we also get
$$
g_k(|\nabla v_k|)\to |\nabla v|^{p-1} \quad \text{uniformly on compact sets of }\R^n.
$$
Hence, passing to the limit in the weak form of \eqref{Pk}, we obtain that $v$ is a weak solution to
$$
\Delta_p v + b(x_0)v^{q-1} = 0 \quad \text{in }\R^n,\qquad v\ge 0,
$$
but this implies, by Theorem \ref{Liouville}, that $v\equiv 0$ and this contradicts the fact that $v(0)=1$ and this completes the proof in this subcase.

\underline{Subcase 1.2}: Now we assume that $\dist(0,\partial\Omega_k)$ is bounded. Therefore, we may assume without loss of generality that
$$
\dist(0,\partial\Omega_k)\to d<\infty \quad\text{as } k\to\infty.
$$
Following \cite{Zou}, after possibly making a rotation and translation, we have that
$$
\Omega_k\to \R^n_d := \{x\in\R^n\colon x_n>-d\}.
$$

Reasoning exactly as in the previous subcase, we have that $v_k\to v$ in the $C^{1,\beta/2}(B_R(0)\cap\overline{\R^n_d})$ topology for every $R>0$ and that $v$ is a weak solution to
$$
\begin{cases}
\Delta_p v + b(x_0)v^{q-1}=0 & \text{in } \R^n_d\\
v = 0 & \text{on } \partial\R^n_d\\
v\ge 0
\end{cases}
$$
Applying now Theorem \ref{Liouville} for the half space, we obtain that $v\equiv 0$ and this contradicts again the fact that $v(0)=1$. This completes the proof in this subcase and hence the proof of case 1.

\underline{Case 2}: $f(M_k)/\lambda_k$ is bounded.

In this case it is immediate to see that  $\lambda_k\to\infty$, and we take the scale factor $N_k$ as
$$
N_k=f^{-1}(\lambda_k).
$$
Hence, $\mu_k=1$ for every $k\in\N$ in this case.

Since by hypothesis $f(M_k)/\lambda_k$ is bounded, we get that $M_k/N_k$ is also bounded and hence
$$
0\le v_k\le C.
$$

Next, we take $y_k=0$ and observe that in this case, $\Omega_k\to\R^n$ and hence, arguing exactly as in the previous case, $v_k\to v$ in $C^{1,\beta/2}_{loc}(\R^n)$ and $v$ is a weak solution to
$$
\begin{cases}
\Delta_p v + b(x_0)v^{q-1} + 1 = 0 & \text{in }\R^n\\
v\ge 0.
\end{cases}
$$
But Theorem \ref{Liouville} for this problem says that there is no nonnegative solution to this equation and hence the proof of Case 2 is complete.
\end{proof}

\section{Existence}\label{Existencia}

In our last section, we prove the existence of solutions to \eqref{P}. Observe that, in general, when $B$ depends on $\nabla u$, variational methods cannot be applied to address the existence problem \eqref{P}. Therefore, a priori bound estimates and a fixed point theorem become crucial.

\begin{teo} \label{te.existence}
Let $\Omega\subset \R^n$ be a bounded domain with $C^{1,\beta_0}$ boundary. Suppose that all conditions of Theorem \ref{teo.apriori} are satisfied. Furthermore, assume that $B$ satisfies the positivity and superlinearity conditions:

There exists $L>0$ such that
\begin{equation*}\tag{P}
		B(x,t,\p)+Lg(t)\geq0,\quad (x,t,\p)\in \Omega\times\R\times \R^n.
\end{equation*}
\begin{equation*}\tag{S}
		B(x,t,\p)+Lg(t)=o(g(t)+g(|\p|)), \quad\text{as $(t,\p)\to 0$ uniformly on $\Omega$}.
\end{equation*}
	 Then \eqref{P} has a positive weak solution $u\in W^{1, G}_0(\Omega)$.
\end{teo}

The main idea to prove Theorem \ref{te.existence} is to use the following theorem from \cite{Krasnoselskii}, which deals with the existence of fixed points on compact operators defined in a cone.

\begin{lema}[Fixed point theorem]\label{fixed point}
	Let $\mathcal{C}$ be a cone in a Banach space $X$  and $\Lambda\colon \mathcal{C}\to \mathcal{C}$ be a compact operator such that $\Lambda(0)=0$. Assume that there exists $r>0$, satisfying:
	\begin{enumerate}
		\item $u\not=t\Lambda(u)$ for all $\|u\|=r$, $t\in [0,1]$.
		
Assume also that there exist a compact homotopy $H\colon [0,1]\times \mathcal{C}\to\mathcal{C}$ and $R>0$ such that:
		\item $\Lambda(u)=H(0,u)$ for all $u\in \mathcal{C}$.
		\item $H(t,u)\not=u$ for any $u$ such that $\|u\|=R$, $t\in [0,1]$.
		\item $H(1,u)\not=u$ for any $u$ such that $\|u\|\leq R$. 
	\end{enumerate}
Then $\Lambda$ has a fixed point in $D=\{u\in \mathcal{C}\colon  r\leq \|u\|\leq R\}$.
\end{lema}

\begin{remark}
In the course of the proof of Theorem \ref{te.existence}, we will need to use the inequality
\begin{equation}\label{young}
g(t)s\le C (G(t)+G(s)).
\end{equation}

In fact, we have that
$$
g(t)s\le \tilde G(g(t)) + G(s),
$$
and \eqref{young} follows from this and \cite[Lemma 2.9]{MR3952156}
\end{remark}

Finally, we will provide a proof for Theorem \ref{te.existence}.
\begin{proof}
	
Let $X:= C^1(\overline{\Omega})$ and let $\mathcal{C}$ be the cone of nonnegative functions, $\mathcal{C}=\{u \in X\colon u \geq 0\}$.
	
Let $T\colon C^1(\overline{\Omega})\to C(\overline{\Omega}) $ be the operator defined by  $T(u)=B(x,u,\nabla u)+Lg(u)$.

Observe that, for each $\psi \in C(\overline{\Omega})$ the problem
$$ 
\begin{cases}
-\Delta_g u+Lg(u)=\psi(x)&\qquad \Omega\\
u=0&\qquad\partial\Omega,
\end{cases}
$$
has unique weak solution $u_\psi\in C^{1,\beta}(\overline{\Omega})$ for some positive $\beta$ \cite[Theorem 1.7]{Lieberman}.

 Therefore, we define $S\colon C(\overline{\Omega})\to C^{1,\beta}(\overline{\Omega})\cap C_0(\Omega)$ as the solution operator, $S(\psi)=u_\psi$, note that $S$ is a continuous and positive operator.

Now, if we denote 
$$
\Lambda=i\circ S\circ T:C^1(\overline{\Omega})\to C^{1,\beta}(\overline{\Omega})\hookrightarrow C^1(\overline{\Omega}),
$$
where $i\colon C^{1,\beta}(\overline{\Omega})\hookrightarrow C^1(\overline{\Omega})$ is the inclusion operator, then $\Lambda$ is a compact operator such that $\Lambda(0)=0$. Observe that the positivity assumption (P) implies that $\Lambda\colon \mathcal{C}\to\mathcal{C}$.
 
 To apply Lemma \ref{fixed point} we must verify conditions (1)--(4) for $\Lambda$.
 
  Assume that $u=t\Lambda(u)$ for some $u\in \mathcal{C}$ such that $\|u\|=r$ and a certain $t\in [0,1]$. Then
 
 $$
 -\Delta_g\left(\frac{u}{t}\right)+Lg\left(\frac{u}{t}\right)=B(z,u,\nabla u)+Lg(u).
 $$
By taking $u$ as a test function, and the (S) hypothesis on $B$ we obtain
\begin{align*}
\int_\Omega g\left(\left|\nabla\left(\frac{u}{t}\right)\right|\right)|\nabla u| + Lg\left(\frac{u}{t}\right)u\,dx &=\int_\Omega u[B(z,u,\nabla u)+Lg(u)]\,dx\\
&= \int_\Omega o(g(u) + g(|\nabla u|)) u \, dx\\
&=\int_\Omega o(G(u)+G(|\nabla u|))\,dx
\end{align*}
as $\|u\|\to0$ where we used \eqref{young} in the last step. But, this implies that
$$
p^- \int_\Omega G(|\nabla u|) + L G(|u|)\, dx \le t^{p^+-1}\int_\Omega o(G(|u|) + G(|\nabla u|))\, dx.
$$
Therefore, we can chose $r>0$ small enough such that the equation $u=t\Lambda(u)$ has no positive solution in $B_r(0)\setminus\{0\}$ for all $t\in [0,1]$.

 By Theorem \ref{teo.apriori}, there exists a positive constant $\lambda_0$ such that \eqref{P.lambda} has no solution. Therefore, we define $H:[0,1]\times \mathcal{C}\to \mathcal{C}$ as
$$
 H(t,u)=i\circ S(T(u)+t\lambda_0).
$$

Clearly $H(0,u)=\Lambda(u)$ for any $u\in\mathcal{C}$, so (2) holds.
Observe that the equation $u=H(t,u)$ is equivalent to
$$
\begin{cases}
	-\Delta_gu+Lg(u)=B(z,u,\nabla u)+Lg(u)+t\lambda_0 &\quad\mbox{in}\quad \Omega\\
	u=0&\quad\mbox{on}\quad \partial\Omega.
\end{cases}
$$
Which is equivalent to
$$
\begin{cases}\label{segunda.equiv}
	\Delta_gu+B(z,u,\nabla u)+t\lambda_0=0 &\quad \mbox{in}\quad \Omega\\
	u=0&\quad \mbox{on}\quad \partial\Omega.
\end{cases}
$$
Hence, by Theorem \ref{teo.apriori}, we have that $
\|u\|_\infty + t\lambda_0 \le C,
$ then choosing $R=C+1$ we obtain that (3) holds.

Finally, $H(1,u)=u$ has not solution in view of the choice of $\lambda_0=C+1$, therefore (4) holds.
In conclusion, $\Lambda$ has a fixed point $u \in \mathcal{C}$ which it is a non-negative solution of \eqref{P} as we wanted to prove. 
\end{proof}

\section*{Acknowledgments}
This work was partially supported by UBACYT Prog. 2018 20020170100445BA, CONICET PIP 11220210100238CO and
ANPCyT PICT 2019-03837 and PICT 2019-00985. J. Fern\'andez Bonder and A. Silva are members of CONICET and I. Ceresa-Dussel is a doctoral fellow of CONICET.

\bibliography{References.bib}

\bibliographystyle{plain}

\end{document}